\documentclass{birkjour}

\usepackage{amsmath}
\usepackage{amssymb}
\usepackage{amsthm}
\usepackage[usenames]{color}

\newtheorem{theorem}{Theorem}[section]
\newtheorem{lemma}[theorem]{Lemma}
\newtheorem{corollary}[theorem]{Corollary}
\theoremstyle{definition}
\newtheorem{example}[theorem]{Example}
\newtheorem{definition}[theorem]{Definition}
\newtheorem{remark}[theorem]{Remark}

\numberwithin{equation}{section}

 \allowdisplaybreaks
\begin{document}

 \title[Sturm-Liouville problems on graphs]{Sturm-Liouville problems on graphs with \\ Robin boundary conditions}
 
\author{Yuri Latushkin}\address{University of Missouri, Columbia, USA}\email{latushkiny@missouri.edu}
\author{Vyacheslav Pivovarchik}\address{South Ukrainian National Pedagogical University, Odesa, Ukraine, and\\ University of Vaasa, Finland} \email{vpivovarchik@gmail.com}
\author{Alesia Supranovych}\address{South Ukrainian National Pedagogical University, Odesa, Ukraine} \email{ghgufgchc@gmail.com}



\date{\today}

\dedicatory{To the memory of H.\ Langer}

\begin{abstract}  We study characteristic functions for the spectral 
Sturm -- Liouville problem with real valued $L_2$-potential on equilateral compact connected simple graph equipped with Robin-Kirchhoff ($\delta$-vertex) boundary conditions. Assuming, in addition, that the graph is a tree and the potential is zero we give some results on asymptotics of  eigenvalues. Also, under these assumptions, we show how to recover coefficients related to the Robin conditions for the quantum graph provided that the shape of the tree and some Robin eigenvalues are known.
\end{abstract}

\maketitle

\newcommand{\im}{\operatorname{Im}}
\newcommand{\re}{\operatorname{Re}}
\newcommand{\ki}{k^{-1}}
\newcommand{\sql}{\sqrt{\lambda}}

\newcommand{\wtp}{\widetilde{\psi}}
\newcommand{\whp}{\widehat{\psi}}
\newcommand{\C}{{\bf C}}
\newcommand{\CB}{{\cal B}}
\newcommand{\CH}{{\cal H}}
\newcommand{\CL}{{\cal L}}
\newcommand{\cotanh}{\mbox{cotanh}}
\newcommand{\HB}{{\cal HB}}
\newcommand{\N}{{\bf N}}
\newcommand{\R}{{\bf R}}
\newcommand{\sgn}{\mbox{sgn}}
\newcommand{\SHB}{{\cal SHB}}
\newcommand{\Z}{{\bf Z}}
\newcommand{\za}{\alpha}
\newcommand{\zb}{\beta}
\newcommand{\ze}{\varepsilon}
\newcommand{\zf}{\varphi}
\newcommand{\zg}{\gamma}
\newcommand{\zk}{\kappa}
\newcommand{\zl}{\lambda}
\newcommand{\zo}{\omega}
\newcommand{\zs}{\sigma}
\newcommand{\zz}{\zeta}
\newcommand{\rf}[1]{(\ref{#1})}
\newcommand{\p}{^{\prime}}
\newcommand{\pp}{^{\prime\prime}}
%
%


\section{Introduction.}

A differential operator is a synthesis of three fairly independent objects: 
(1) the set of support of functions on which the differential operator is acting (say, a finite interval, a graph of certain shape, a multidimensional domain, or a manifold), (2) the differential operation itself (determined, say, by the potential in the Sturm-Liouville problems), and (3) the domain of the operator acting in a suitably chosen function space (which is usually determined by boundary conditions). In this paper we will deal with Sturm-Lioville  problems on graphs, and in this context, respectively, one anticipates three fairly independent \textit{inverse problems} to be considered: Given the spectrum (or any specific spectral data) of the Sturm-Lioville operator, find: (1) the shape of the graph, (2) the potential, (3) the parameters in the boundary conditions. In recent years problems (1) and (2) have been addressed in a vast amount of sources and attracted attention of many. However, problem (3) to the best of our knowledge was considered only by a very few authors, and it is the main objective of this paper to attract attention to this topic.

We begin with a brief account of the history of Sturm-Liouville inverse problems.  For  a \textit{finite interval} it starts with the celebrated Ambarzumians Theorem \cite{A}, continues at  Borg's paper  \cite{Bo}, and culminates in the classical work on the solution of the inverse problem given by Levitan and Gasymov \cite{LG} and Marchenko \cite{Ma}. 
The inverse Sturm-Liuville problems on \textit{metric graphs} were first considered by Gerasimenko \cite{G} and Exner and Seba \cite{ES} who constructed the potential of the Sturm-Liouville equation using given spectral data. This marks the beginning of research on problem (2) from the above mentioned list of inverse problems on graphs. The analysis of the inverse problems of type (2) was further continued  in \cite{ALTW, BW, CP, D, GW, P1, P2, P3,RS, VMS}.
The work on the inverse problems of type (1) began by Gutkin and Smilansky in \cite{GS} and by von Below in \cite{Be}. 
Gutkin  and Smilansky showed that if the lengths of the edges are non-commensurate then the spectrum of the quantum graph (that is, the second order derivative operator on the graph with standard boundary conditions) uniquely determines the shape of the graph. Von Below considered the opposite case and showed that there exist co-spectral graphs, i.e., non-isomorphic graphs with the same spectrum of a Sturm-Liouville problem.   
Many interesting results on classes of cospectral graphs can be found in \cite{BPB}, \cite{BG}, \cite{FLP}. In \cite{KN} (see also \cite{BKS}) a ``geometric" Ambarzumian Theorem was proved stating that the unperturbed spectrum uniquely determines the shape of a   $P_2$ graph.   In \cite{ChP} this result was generalized to the case of simple connected equilateral  graphs with the number of vertices not exceeding five and to the case of trees with the number of vertices not exceeding eight. An example of co-spectral trees with nine vertices was given in \cite{Pis}.

The inverse problems of type (3) on graphs appear to be significantly  less popular. In general, the problem could be formulated as follows: Suppose that the shape of a graph is known, that the potentials on the edges are known (in particular, they may be known to be identically zero), and we are given a finite number of eigenvalues and need to find unknown constants in the boundary and matching conditions.  We are aware of only very few papers on this topic: \cite{C} and \cite{AKN}, and very recent work in \cite{KKM} and \cite{PS25}. 

 In the current paper we consider a general Robin problem  on compact graphs with $p$ vertices obtained by replacing the standard Kirchhoff conditions $\sum_j y'_j(\ell)=\sum_k y_k'(0)$ at a vertex $v_i$ with incoming edges $e_j$ and outgoing edges $e_k$ by the Robin-Kirchhoff (more often called $\delta$-vertex) conditions $\sum_j y'_j(\ell)+b_iy_{j_1}(\ell)=\sum_k y_k'(0)$ with some real coefficients $b_i$, $i=1,\ldots,p$.
For trees, we study asymptotics of the eigenvalues of the respective Sturm-Liouville problems  and show how to recover certain constants $B_{r,l}$ related to the coefficients  $b_i$ in the Robin-Kirchhoff boundary conditions using quantum graph eigenvalues (that is, we are assuming that the potentials are identically zero).

\subsubsection*{Acknowledgements}
This material is based upon work supported by the US NSF grants DMS-2108983/2106157; the authors thank the US NSF, National Academy of Science and Office of Naval Research Global for the support  of the project  ``IMPRESS-U: Spectral and geometric methods for damped wave equations with applications to fiber lasers".
The hospitality of the Institute of Mathematics of the Polish Academy of Sciences where the work began is gratefully acknowledged. VP was partially supported by the Academy of Finland (project no. 358155) and  is grateful to the University of Vaasa for hospitality. AS and VP thank the Ministry of Education and Science of Ukraine for the support in completing the work on the project 'Inverse problems of finding the shape of a graph by spectral data', state registration number 0124U000818 and on the project “Spectral and scattering problems for differential equations of quantum mechanics on graphs”, state registration number 0126U000898.  YL would
like to thank the Courant Institute of Mathematical Sciences at NYU
and especially Prof. Lai-Sang Young for their hospitality.

Our special thanks go to the anonymous referee who pointed out several essential deficiencies in the original version of the manuscript.   
 
\section{The Robin problems and their characteristic functions}


Let $G$ be an equilateral compact connected simple graph with $p$ vertices $v\in\mathcal{V}$, $\text{card}(\mathcal{V})=p$, and $g$ edges $e_j\in\mathcal{E}$, $j=1,\ldots,g$; each of the edges is of length $\ell$.  The orientation of the edges is arbitrary. For each $v\in\mathcal{V}$ we denote by $d(v)$ its degree, by $d_{\textrm{in}}(v)$ the number of incoming and by $d_{\textrm{out}}(v)$ the number of outgoing edges so that $d_{\textrm{in}}(v)+d_{\textrm{out}}(v)=d(v)$. We enumerate the vertices $v_1,\ldots,v_p\in{\mathcal{V}}$ arbitrary. 

We consider  the Sturm-Liouville equations on the edges, 
\begin{equation}
\label{2.3}
-y_j''+q_j(x)y_j=\lambda y_j, \, x\in e_j=[0,\ell],\,  j=1,2,\ldots,g, 
\end{equation} 
where $q_j\in L_2(0,\ell)$ are real.
Given real numbers $b_1,\ldots, b_p$, we set up the  \textit{Robin spectral  problem} on the graph $G$ by imposing the generalized (standard) Robin (-Kirchhoff)  conditions at the vertices $v_i$, $i=1, \ldots, p$, as follows:
If $v_i$ is a pendant vertex  then
for each edge $e_j$ incoming to (respectively,  for each edge $e_k$ outgoing from) 
$v_i$ we set
\begin{equation}
\label{2.4}
y_j'(\ell)+b_iy_j(\ell)=0\,  \text{ (respectively, $ -y_k'(0)+b_iy_k(0)=0$)}. 
\end{equation}
At  each interior vertex $v_i$, $i=1,\ldots,p$, we impose the continuity conditions   
\begin{equation}
\label{2.5}\begin{split}
y_{j_1}(\ell)-y_{j_2}(\ell)&=0,\dots, y_{j_{d_{\textrm{in}}}-1}(\ell)
-y_{j_{d_{\textrm{in}}}}(\ell)=0,\\ y_{j_{d_{\textrm{in}}}}(\ell)-y_{k_1}(0)&=0,\dots, y_{k_{d_{\textrm{out}}}-1}(0)-y_{k_{d_{\textrm{out}}}}(0)=0\end{split}
\end{equation}
for all incoming to $v_i$ edges $e_j$ and for all edges $e_k$ outgoing from $v_i$,  and the  Robin-Kirchhoff's conditions
\begin{equation}
\label{2.6}
\mathop{\sum}\nolimits_j y'_j(\ell)+b_iy_{j_1}(\ell)-\mathop{\sum}\nolimits_k y_k'(0)=0
\end{equation}
where the first sum is taken over all edges $e_j$ incoming to $v_i$ while the second sum is taken over all edges $e_k$ outgoing from $v_i$, and $e_{j_1}$ is any one of the incoming edges. 
Clearly, conditions  (\ref{2.4})--(\ref{2.6}) with $b_i=0$ are the so called \textit{standard} conditions, cf.\ \cite{BK,Kuraso,MP2}. Also, if we allow $b_i=\infty$ then, formally, conditions (\ref{2.4})--(\ref{2.6}) correspond to the Dirichlet condition
\eqref{2.7} below.


A more standard terminology for the boundary conditions \eqref{2.4}--\eqref{2.6} is the \textit{$\delta$-vertex} or \textit{$\delta$-type} condition, or \textit{$\delta$-coupling}, cf.\ \cite{BK,ES,Kuraso,MP2}.

Given the Robin problem \eqref{2.4}--\eqref{2.6} at all vertices $\{v_1,\ldots,v_p\}$ (where we allow some or even all constants $b_i$ to be equal to zero), we will set up the following auxiliary Dirichlet-standard spectral  problems on the graph $G$:  
For any given $r\in\{1,\ldots, p\}$ and any integers $1\le i_1<i_2<\ldots<i_{r}\le p$ and the corresponding vertices $v_{i_1},\ldots,v_{i_{r}}$ we impose, at each of the vertices $v_{i_l}\in\{v_{i_1},\ldots v_{i_r}\}$, $l=1,\ldots,r$,
 the generalized Dirichlet conditions,
\begin{equation}
\label{2.7} 
y_j(\ell)=0\, \text{ (respectively, } y_k(0)=0) 
\end{equation}
for all edges $e_j$ incoming to (respectively, for all edges
$e_k$ outgoing from) the vertex $v_{i_l}$.  
 We impose the standard conditions  (with $b_i=0$) at all remaining vertices in the set $\{v_1,\ldots,v_p\}\setminus\{v_{i_1},\ldots,v_{i_{r}}\}$. We call the problem \textit{Dirichlet-standard auxiliary $\{i_1i_2...i_{r}\}$-problem}.

In order to define the characteristic functions of the Robin  problem and of the Dirichlet-standard auxiliary problems  
we look for real coefficients $\alpha_1, \beta_1, \ldots, \alpha_g, \beta_g$ such that the  solution of \eqref{2.3} can be expressed in the form 
\[
 y_j(x)=\alpha_js_j(\lambda,x)+\beta_jc_j(\lambda,x), \,   x\in e_j=[0,\ell], j=1,\ldots, g,
\]
 where $s_j(\lambda,x)$ is the solution of \eqref{2.3} which satisfies the conditions $s_j(\lambda,0)=s_j'(\lambda,0)-1=0$ and $c_j(\lambda,x)$ is the solution of \eqref{2.3} which satisfies $c_j(\lambda,0)-1=c_j'(\lambda,0)=0$.
Substituting this into the vertex conditions at each vertex, 
we obtain  systems of $2g$ linear algebraic equations with $2g$ unknowns $\alpha_1,\beta_1,\ldots, \alpha_g, \beta_g$. Indeed, each vertex $v$ generates $d(v)-1$ equations coming from the continuity conditions and one more equation coming from the Kirchhoff or the Dirichlet condition, and we recall the well known ``hand-shake'' relation $2g=\sum_{v\in\mathcal{V}}d(v)$.

We denote  by $\Phi(\lambda, b_1,...,b_p)$ the $2g\times 2g$-matrix of the system corresponding to the Robin problem and by  $\Phi_{i_1i_2...i_{r}}(\lambda)$ the matrix corresponding to a Dirichlet-standard auxiliary $\{i_1i_2... i_{r}\}$-problem, and
observe that the matrices involve the values $s_j(\lambda,\ell)$, $s'_j(\lambda,\ell)$, $c_j(\lambda,\ell)$, $c'_j(\lambda,\ell)$, $j=1,\ldots,g$, and $b_i$, $i=1,\ldots,p$. 


\begin{remark}\label{rem:UNPHI}
We stress that the orientation of the edges and the enumeration of the edges and vertices are assumed to be fixed once and for all. This determines the order of columns and rows of the matrix $\Phi$. Moreover, even a slight (equivalent) change in the boundary conditions (say, writing the Dirichlet condition as $-y_j(\ell)=0$ in place of $y_j(\ell)=0$) will change $\Phi$; we therefore assume that the boundary conditions are used exactly as written in \eqref{2.4}--\eqref{2.7}.
\hfill$\Diamond$\end{remark}

\begin{definition}\label{defphi}
We call
$
\lambda\mapsto \phi(\lambda, b_1,b_2,...,b_p):=\det(\Phi(\lambda, b_1,b_2,..., b_p))$
 the \emph{characteristic
function} of the Robin problem. For any  $r\in\{1,\ldots,p\}$ and any integers
$1\le i_1<\ldots<i_{r}\le p$ we call
$
\lambda\mapsto \phi_{i_1i_2...i_{r}}(\lambda):=\det(\Phi_{i_1i_2...i_{r}}(\lambda))
$
 the \emph{characteristic
function} of the Dirichlet-standard auxiliary $\{i_1i_2...i_{r}\}$-problem.
\end{definition}

We are ready to present the first main result of the section where we express $\phi(\lambda, b_1,b_2,...,b_p)$, the Robin characteristic function, via $\phi_{i_1i_2...i_{r}}(\lambda)$, the Dirichlet-standard characteristic functions, exposing dependance on the Robin constants $b_i$'s explicitly.


\begin{theorem}\label{Theorem 2.2} Let $G$ be an equilateral compact connected simple graph and $q_j\in L_2(0,\ell)$ are real valued potentials. Then 
\begin{equation}
\label{2.10}
\begin{split}
&\phi(\lambda, b_1,b_2,..., b_p)=\phi(\lambda, 0,0, ..., 0)+
\sum\nolimits_{i=1}^p b_i\phi_{i}(\lambda)\\
&\quad+\sum\nolimits_{1\leq i_1<i_2\leq p}b_{i_1}b_{i_2}\phi_{i_1i_2}(\lambda)+
\ldots + \big(\prod\nolimits_{i=1}^p b_i\big) \phi_{12...p}(\lambda).
\end{split}\end{equation}
\end{theorem}
\begin{proof}

The matrix $\Phi(\lambda, b_1, 0, 0, ..., 0)$ differs from $\Phi(\lambda, 0 , 0, 0, ..., 0)$ in the row corresponding to the Kirchhoff condition at the vertex $v_1$.
Indeed, let us denote by $e_{j_1},\ldots,e_{j_{d_\textrm{in}(v_1)}}$ the incoming to  and by $e_{k_1},\ldots,e_{k_{d_\textrm{out}(v_1)}}$ the outgoing from the vertex $v_1$ edges. For definiteness, we assume that each of $d_\textrm{in}(v_1)$ and $d_\textrm{out}(v_1)$ is at least one.
The Kirchhoff condition at $v_1$ for the Robin problem then gives
\begin{equation*}\begin{split}
&\sum\nolimits_{m=1}^{d_\textrm{in}(v_1)}\big(\alpha_{j_m}s'_{j_m}(\lambda,\ell)
+\beta_{j_m}c'_{j_m}(\lambda,\ell)\big)\\&\quad+b_1\big(\alpha_{j_1}s_{j_1}(\lambda,\ell)+
\beta_{j_1}c_{j_1}(\lambda,\ell)\big)-\sum\nolimits_{n=1}^{d_\textrm{out}(v_1)}\alpha_{k_n}=0.\end{split}\end{equation*}
On the other hand, setting $r=1$, $i_1=1$, the Dirichlet condition at $v_1$ for the Dirichlet-standard auxiliary $\{1\}$-problem gives $\alpha_{j_1}s_{j_1}(\lambda,\ell)+
\beta_{j_1}c_{j_1}(\lambda,\ell)=0$ (recall that for the latter problem the Dirichlet condition is imposed for all edges incident to $v_1$ and thus the continuity conditions at $v_1$ automatically holds).
 Then the row of the matrix $\Phi(\lambda, b_1,0,\ldots,0)$ corresponding to the Kirchhoff condition at $v_1$ is of the form 
\begin{equation*}\begin{split}
&\big\{\ldots 0,  s'_{j_1}(\lambda,\ell)+b_1s_{j_1}(\lambda,\ell), c'_{j_1}(\lambda,\ell)+b_1c_{j_1}(\lambda,\ell),0\ldots 0, \\
&\qquad\qquad s'_{j_2}(\lambda,\ell), c'_{j_2}(\lambda,\ell), 0 \ldots 0, s'_{j_{d_{\textrm{in}}(v_1)}}(\lambda,\ell),c'_{d(v_1)}(\lambda,\ell), \\
&\qquad\qquad\qquad\qquad 0\ldots 0, -1,0,\ldots, -1,0,\ldots  0 \ldots 0\big\}\\
&=\big\{\ldots 0,  s'_{j_1}(\lambda,\ell), c'_{j_1}(\lambda,\ell),0\ldots 0, \\
&\qquad\qquad s'_{j_2}(\lambda,\ell), c'_{j_2}(\lambda,\ell), 0 \ldots 0, s'_{j_{d_{\textrm{in}}(v_1)}}(\lambda,\ell),c'_{d(v_1)}(\lambda,\ell), \\
&\qquad\qquad\qquad\qquad 0\ldots 0, -1,0,\ldots, -1,0,\ldots  0 \ldots 0\big\}\\
&+
b_1\big\{\ldots 0, s_{j_1}(\lambda,\ell), c_{j_1}(\lambda,\ell),0,0,...,0\big\},
\end{split}\end{equation*}
where the curly brackets in the last line give the row of the matrix $\Phi_1(\lambda)$ for the Dirichlet-standard auxiliary $\{1\}$-problem that corresponds to the Dirichlet condition at $v_1$ for the edge $e_{j_1}$.
Thus,  
\begin{equation*}\begin{split}
\det\big(\Phi(\lambda, b_1,0,0...,0)\big)&=\det \big(\Phi(\lambda,0,0,...,0)\big)+b_1\det\big(\Phi_1(\lambda)\big), \text{  
that is,}\\
\phi(\lambda, b_1,0,...,0)&=\phi(\lambda, 0,0, ..., 0)+
b_1\phi_1(\lambda).\end{split}\end{equation*}
Repeating this procedure inductively we finally arrive at (\ref{2.10}). Indeed, the relation between the rows of $\Phi(\lambda,b_1,0,\ldots,0)$ and $\Phi_1(\lambda)$ can be written as $\Phi^{(v_1)}(\lambda,b_1,0,\ldots,0)=\Phi^{(v_1)}(\lambda,0,\ldots,0)+b_1\Phi^{(v_1)}_1(\lambda)$ where $\Phi^{(v_1)}$ is the row of the matrix $\Phi$ corresponding to the vertex $v_1$. A similar argument yileds
$\Phi^{(v_2)}(\lambda,b_1,b_2,0,\ldots,0)=\Phi^{(v_2)}(\lambda,b_1,0,\ldots,0)+b_2\Phi^{(v_2)}_2(\lambda, b_1)$ making the induction possible.
\end{proof}

A similar results holds if, from the very beginning, some of the (standard) Robin(-Kirchhoff) conditions in the Robin problem  
are replaced by the Dirichlet conditions. That is, for any $p'<p$  we fix vertices $\{v_1,\ldots,v_{p'}\}$ where the Robin conditions are imposed, and at all  other $p-p'$ vertices $v_{p'+1},\ldots, v_{p}$ we impose the Dirichlet conditions (in other words, we choose $b_1\ge0,\ldots, b_{p'}\ge0$ while $b_{p'+1}=\ldots=b_{p}=\infty$). We call this new boundary value problem the \textit{Robin-Dirichlet problem}. One can consider then the Dirichlet-standard auxiliary $\{i_1\ldots i_{r}\}$-problem as above but for the new Robin-Dirichlet problem instead of just the Robin problem.






 



\begin{corollary}\label{Corollary 2.3} Let us assume that $p'<p$ and  impose the Robin conditions at $p'$  vertices $\{v_1,\ldots,v_{p'}\}$ and the Dirichlet conditions at the remaining  $p-p'$ vertices $\{v_{p'+1},\ldots,v_p\}$. Then
\begin{equation}
\label{2.11}\begin{split}
&\phi^{(p')}(\lambda, b_1,b_2,..., b_{p'})=\phi^{(p')}(\lambda, 0,0, ..., 0)+\sum\nolimits_{i=1}^{p'} b_i\phi^{(p')}_{i}(\lambda)\\&\quad+\sum\nolimits_{1\leq i_1<i_2\leq p'}b_{i_1}b_{i_2}\phi^{(p')}_{i_1i_2}(\lambda)+
... + \big(\prod\nolimits_{i=1}^{p'} b_i\big) \phi^{(p')}_{i_1i_2\ldots i_{p'}}(\lambda),
\end{split}\end{equation}
where $\phi^{(p')}(\lambda, b_1,b_2,...,b_{p'})$ is the characteristic function of problem \eqref{2.3}-\eqref{2.7} with the Robin conditions at the vertices $v_1$, $v_2$,...,$v_{p'}$, the Dirichlet conditions at the remaining  $p-p'$ vertices  $v_{p'+1}$, $v_{p'+2}$,...,$v_{p}$, while $\phi^{(p')}_{i_1\dots i_r}(\lambda)$ for any $1\le r\le p'$ and any integers $1\le i_1<\ldots< i_r\le p'$ are the characteristic function of the auxiliary $\{i_1\ldots i_{r}\}$-problem. \end{corollary}

\begin{example}\label{Example 2.4}
Formulas above can be used to construct characteristic functions for any graphs but they are especially simple for the star-graphs. As an example, we consider a star graph with three edges of length $\ell$ oriented toward the centre. We denote by $v_1,v_2,v_3$ the pendant and by $v_4$ the centre vertices and impose the Robin boundary condition $-y'_i(0)+b_iy_i(0)=0$, $i=1,2,3$, at the pendant vertices and the (standard) Kirchhoff boundary conditions $y_1(\ell)=y_2(\ell)$, $y_2(\ell)=y_3(\ell)$, $y'_1(\ell)+y_2'(\ell)+y'_3(\ell)=0$ at the centre $v_4$, that is, we assume that $b_4=0$.
The matrix $\Phi(\lambda, b_1,b_2,b_3,0)$ has the form
\begin{equation}\label{BigPhi}
\begin{bmatrix}
-1&b_1&0&0&0&0\\
\\
0&0&-1&b_2&0&0\\
\\
0&0&0&0&-1&b_3\\
s_1(\lambda,\ell)&c_1(\lambda,\ell)&-s_2(\lambda,\ell)&-c_2(\lambda,\ell)&0&0\\
\\
0&0&s_2(\lambda,\ell)&c_2(\lambda,\ell)&-s_3(\lambda,\ell)&-c_3(\lambda,\ell)\\
\\
s_1'(\lambda,\ell)&c_1'(\lambda,\ell)&s_2'(\lambda,\ell)&c_2'(\lambda,\ell)&s_3'(\lambda,\ell)&c_3'(\lambda,\ell)
\end{bmatrix}.
\end{equation}
To obtain the characteristic matrix $\Phi_1(\lambda)$ of the $\{1\}$-auxiliary problem from $\Phi(\lambda, b_1,b_2,b_3,0)$, we set in  \eqref{BigPhi} $b_1=b_2=b_3=0$ and replace $\{-1\,\,\, 0\}$ by $\{0\,\,\, 1\}$ in the first row  as this corresponds to imposing the Dirichlet condition at $v_1$ instead of the Neumann condition; similar changes needed to build $\Phi_i(\lambda)$ for $i=2,3$. To construct $\Phi_{12}(\lambda)$, the characteristic function of the $\{1,2\}$-auxiliary problem, we set in  \eqref{BigPhi} $b_1=b_2=b_3=0$ and replace $\{-1\,\,\, 0\}$ by $\{0\,\,\, 1\}$ in the first and in the second rows, and so on. Consecutive expansion of the respective determinants using their first rows gives \begin{align*}
\phi(\lambda,0,0,0,0)&=\det\begin{bmatrix}c_1(\lambda,\ell)&-c_2(\lambda,\ell)&0\\0&c_2(\lambda,\ell)&-c_3(\lambda,\ell)\\c_1'(\lambda,\ell)&c_2'(\lambda,\ell)&c_3'(\lambda,\ell)\end{bmatrix}\\&=\big(c_1(\lambda,x)c_2(\lambda,x)c_3(\lambda,x)\big)'\big|_{x=\ell}\\
\phi_1(\lambda)&=\big(s_1(\lambda,x)c_2(\lambda,x)c_3(\lambda,x)\big)'\big|_{x=\ell},\ldots, \\\phi_{12}(\lambda)&=\big(s_1(\lambda,x)s_2(\lambda,x)c_3(\lambda,x)\big)'\big|_{x=\ell},\ldots\\
\phi_{123}(\lambda)&=\big(s_1(\lambda,x)s_2(\lambda,x)s_3(\lambda,x)\big)'\big|_{x=\ell},
\end{align*} 
where prime stands for the $x$-derivative of the products of the respective functions.
The calculation leading to the formula reveals that passing from $\Phi$ to $\Phi_1$ requires replacing $c_1$ by $s_1$, passing from $\Phi$ to $\Phi_2$ requires replacing $c_2$ by $s_2$,  and so on.

Similarly, for a star graph with pendant vertices $v_1,\ldots,v_g$ and $g\ge1$ edges oriented toward the centre $v_{g+1}$ equipped with the standard conditions at the centre and Robin conditions at the pendant vertices we obtain the following formula,
\begin{equation*}\begin{split}
\phi&(\lambda, b_1,\ldots,b_g,0)=\big(c_1(\lambda,x)\cdot\ldots\cdot c_g(\lambda,x)\big)'\big|_{x=\ell}\\&+\sum\nolimits_{i=1}^gb_i\big(c_1(\lambda,x)\cdot\ldots\cdot s_i(\lambda,x)\cdot\ldots\cdot c_g(\lambda,x)\big)'\big|_{x=\ell}\\&\quad+
\sum\nolimits_{1\le i_1<i_2\le g}^gb_{i_1}b_{i_2}\big(c_1(\lambda,x)\cdot\ldots\cdot s_{i_1}(\lambda,x)\cdot\ldots\cdot  s_{i_2}(\lambda,x)\cdot\ldots\\&\qquad\cdot c_g(\lambda,x)\big)'\big|_{x=\ell}+\ldots+b_1b_2\ldots b_g\big(s_1(\lambda,x)\cdot\ldots\cdot s_g(\lambda,x)\big)'\big|_{x=\ell},
\end{split}
\end{equation*}
where each summand in the first sum contains exactly one function $s_i(\lambda, \cdot)$, in the second sum exactly two functions $s_{i_1}(\lambda,\cdot)$ and $s_{i_2}(\lambda,\cdot)$, and so on.
\hfill$\Diamond$\end{example}

Throughout,  we will use the following notation. We denote by
\begin{equation}
\label{2.1}
 D=\textrm{diag } \{d(v): v\in\mathcal{V}\}
\end{equation} 
 the diagonal $(p\times p)$ degree matrix; here $d(v)$ is the degree of a vertex $v$. 
  Let $A$ be the $(p\times p)$  adjacency matrix of $G$ whose entries $a_{vw}=1$ when $v,w\in\mathcal{V}$ are adjacent vertices and $a_{vw}=0$ when they are not. 
   Let $1\le r\le p$ and $1\le i_1<\ldots<i_r\le p$ be as in Definition \ref{defphi}. We denote by
$A_{i_1i_2...i_{r}}$, respectively, $D_{i_1i_2...i_{r}}$, the principal submatrix of $A$, respectively, $D$ obtained by deleting the rows and
 columns  corresponding to the vertices $v_{i_1},v_{i_2},\ldots,v_{i_{r}}$. 
We introduce notation
\begin{equation}
\label{2.2}\begin{split}
\psi(z)&:=\det(-zD+A), \\ \psi_{i_1i_2\ldots i_{r}}(z)&: =\det(-z D_{i_1i_2\ldots i_{r}}+A_{i_1i_2\ldots i_{r}}), \, \psi_{1\dots p}(z):=1.\end{split}
\end{equation}
The next result follows from  \cite[Theorem 6.4.2]{MP2}  (which, in turn, a generalization of \cite{Be}) adapted to the case of our Robin problem and to our Dirichlet problem.


We follow convention in \cite{BLP} and use determinants of $-zD+A$, $-zD_i+A_i$ in \eqref{2.2} rather than $zD-A$, $zD_i-A_i$ as in \cite{MP2} which explains the presence
 of the pre-factors in \eqref{2.12}--\eqref{2.14bis}.
\begin{theorem}\label{Theorem 2.4} Let $G$ be a connected simple graph  with $p\geq 2$ vertices.  Assume that all edges have the same length $\ell$ and the same  potential $q$ symmetric with respect to the midpoint of each edge so that $q(\ell-x)=q(x)$ for almost all $x\in[0,\ell]$.  Then the spectrum of problem (\ref{2.3})--(\ref{2.7})  coincides with the set of zeros of the function  $\phi(\lambda, b_1,b_2,...,b_p)$. Moreover, the characteristic functions $\phi$'s introduced in Definition \ref{defphi} are related to the determinants $\psi$'s introduced in \eqref{2.2} by the formulas
\begin{align}
\label{2.12} 
&\phi(\lambda, 0,0,...,0)=(-1)^p\left(s(\lambda,\ell)\right)^{g-p}\psi(c(\lambda, \ell)),\\
\label{2.13} 
&\phi_i(\lambda)=(-1)^{p-1}\left(s(\lambda, \ell)\right)^{g-p+1}\psi_i(c(\lambda, \ell)),
\, 1\le i\le p, \\
\label{2.14} 
&\phi_{i_1i_2}(\lambda)=(-1)^{p-2}\left(s(\lambda, \ell)\right)^{g-p+2}\psi_{i_1i_2}(c( \lambda, \ell)),\, 1\le i_1<i_2\le p,\ldots,\\
&\phi_{i_1\ldots i_r}(\lambda)=(-1)^{p-r}\left(s(\lambda, \ell)\right)^{g-p+r}\psi_{i_1\ldots i_r}(c( \lambda, \ell)),\label{2.14bisbis}\\ &\hskip4cm 1\le i_1<\ldots<i_r\le p,\ldots,\nonumber\\
&\phi_{1\dots p}(\lambda)=s(\lambda,\ell)^g.\label{2.14bis}\end{align} 
\end{theorem}
In
 particular, in case of $q=0$ we have
\begin{align}
\label{2.15} 
&\phi(\lambda, 0,0,...,0)=(-1)^p\big({\sin( \sqrt{\lambda} \ell)}/{\sqrt{\lambda}}\big)^{g-p}\psi(\cos \sqrt{\lambda} \ell),\\
\label{2.16} 
&\phi_i(\lambda)=(-1)^{p-1}\big({\sin( \sqrt{\lambda} \ell)}/{\sqrt{\lambda}}\big)^{g-p+1}\psi_i(\cos \sqrt{\lambda} \ell), \, 1\le i\le p, \\
\label{2.17} 
&\phi_{i_1i_2}(\lambda)=(-1)^{p-2}\big({\sin( \sqrt{\lambda} \ell)}/{\sqrt{\lambda}}\big)^{g-p+2}\psi_{i_1i_2}(\cos \sqrt{\lambda}\ell),\\& \qquad\qquad\hskip2cm
1\le i_1<i_2\le p,\ldots,\nonumber\\
\label{2.17bis} 
&\phi_{i_1\ldots i_r}(\lambda)=(-1)^{p-r}\big({\sin( \sqrt{\lambda} \ell)}/{\sqrt{\lambda}}\big)^{g-p+r}\psi_{i_1\ldots i_r}(\cos \sqrt{\lambda}\ell),\\& \qquad\qquad\hskip2cm
1\le i_1<\ldots<i_r\le p,\ldots,\nonumber\\
&\phi_{1\dots p}(\lambda)=\big(\sin(\sqrt{\lambda}\ell) /\sqrt{\lambda}\big)^g.
\end{align} 

We denote by $z_1$, $z_2$ ..., $z_{p-1}$, $z_p$ the zeros of the polynomial $\psi(z)$ from \eqref{2.2}.  If $G$ is a tree then $g-p=-1$ and, as proven in \cite{FCh}, $z_1=-1$ and $z_p=1$ are simple roots; we denote by $m_l$ the multiplicity of the root $z_l$ of $\psi(z)$ for $l=2,\ldots,p-1$. In particular, for the tree $G$ we have 
\begin{align}
\label{2.18} 
&\phi(\lambda, 0,0,...,0)=-(-1)^p\sqrt{\lambda}\sin (\sqrt{\lambda} \ell)\widetilde{\psi}\big(\cos \sqrt{\lambda} \ell)\big),\\
\label{2.19} 
&\phi_i(\lambda)=(-1)^{p-1}\psi_i\big(\cos(\sqrt{\lambda} \ell)\big),\, 1\le i\le p,\\
\label{2.20} 
&\phi_{i_1i_2}(\lambda)=(-1)^{p-2}\big({\sin(\sqrt{\lambda} \ell)}/{\sqrt{\lambda}}\big)\psi_{i_1i_2}\big(\cos(\sqrt{\lambda} \ell)\big), \\&\qquad \qquad\qquad 
1\le i_1<i_2\le p,\ldots,\nonumber\\
\label{2.20bis} 
&\phi_{i_1\ldots i_r}(\lambda)=(-1)^{p-r}\big({\sin(\sqrt{\lambda} \ell)}/{\sqrt{\lambda}}\big)^{r-1}\psi_{i_1\ldots i_r}\big(\cos(\sqrt{\lambda} \ell)\big), \\& \qquad\qquad\qquad 
1\le i_1<\ldots<i_r\le p,\ldots,\nonumber\\
&\phi_{1\dots p}(\lambda)=\big(\sin(\sqrt{\lambda}\ell) /\sqrt{\lambda}\big)^{p-1};\label{2.20new}\end{align}   
in \eqref{2.18}  and in what follows we use notation $\widetilde{\psi}(z)=(z^2-1)^{-1}\psi(z)$.

\begin{lemma}\label{Lemma 2.5} If $G$ is a tree with $p\ge2$ vertices 
 then $\psi_i(\pm 1)=(\mp 1)^{p-1}$ for any $i=1,\ldots,p$ and $\wtp(\pm 1)=(p-1)(\mp1)^{p-2}$.
 \end{lemma} 
 \begin{proof}
 We fix $i$ and re-enumerate the vertices of the tree $G$ so that $v_0:=v_i$ is the root, and by $v_k$, $k=1,\ldots d(v_0)$, we now denote the vertices adjacent to the root. We present the proof for $d(v_0)\ge2$, the case $d(v_0)=1$ is similar and easier. We will use notation explained in \cite{BLP}, that is, we represent $G$ as the union of subtrees $G_k$, $k=1,\ldots,d(v_0)$, having the common root $v_0$, we denote by $\widehat{G}_k$ the subtree obtained from $G_k$ by deleting $v_0$ and the edge connecting $v_0$ and $v_k$ so that $v_k$ is the root of $\widehat{G}_k$, and denote $\psi_{{}_G}(z)=\det(-zD_G+A_G)$, $\widehat{\psi}_{{}_G}(z)=\det(-z\widehat{D}_G+\widehat{A}_G)$, where $\widehat{D}_G$, $\widehat{A}_G$ are obtained from the degree and adjacency matrices $D_G:=D$, $A_G:=A$ by deleting the first row and the first column corresponding to the root $v_0$ of $G$. We let $\wtp_{{}_{G}}(z)=\psi_{{}_G}(z)/(z^2-1)$.
 
 Using this notation, we have to show two equalities: 
 \begin{equation}\label{eqlts}
 \widehat{\psi}_{{}_G}(\pm1)=(\mp1)^{p_{{}_G}-1} \text{ and } \wtp_{{}_G}(\pm1)=(p_{{}_G}-1)(\mp1)^{p_{{}_G}-2},\end{equation} where $p_{{}_G}:=p$ is the number of vertices in the tree $G$. 
We will use induction. The first equality clearly holds for $p_{{}_G}=2$ when $G$ is the segment $[0,\ell]$ of the real line as in this case $-z\widehat{D}_G+\widehat{A}_G=-z$, the $(1\times 1)$ matrix. The second equality also holds for $p=2$ and $p=3$ because by a direct computation of $\det(-zD_{{}_G}+A_{{}_G})$ we have $\psi_{{}_G}(z)=z^2-1$ if $G$ is a segment and $\psi_{{}_G}(z)=-2z(z^2-1)$ if $G$ is a tree with two edges.
 
 To justify the induction steps, we recall the relations
 \begin{align}\label{hathat}
 \widehat{\psi}_{{}_G}(z)&=\prod\nolimits_{k=1}^{d(v_0)}\whp_{{}_{G_k}}(z)=\prod\nolimits_{k=1}^{d(v_0)}\big(\psi_{{}_{\widehat{G}_k}}(z)-z\widehat{\psi}_{{}_{\widehat{G}_k}}(z)\big),\\\label{hathat2}
 \psi_{{}_G}(z)/\whp_{{}_G}(z)&=-zd(v_0)\\& \quad-
 \big(\whp_{{}_{\widehat{G}_1}}(z)/\whp_{{}_{G_1}}(z)\big)-\ldots-\big(\whp_{{}_{\widehat{G}_{d(v_0)}}}(z)/\whp_{{}_{G_{d(v_0)}}}(z)\big)\nonumber
 \end{align}
 proven in \cite{BLP}, see there Remark 2.1, eqn.(2.3) and eqn.(2.9), respectively.
 A reason why \eqref{hathat} and \eqref{hathat2} hold is that the matrix $-z\widehat{D}_G+\widehat{A}_G=\oplus_k\big(-z\widehat{D}_{G_k}+\widehat{A}_{G_k}\big)$ is block-diagonal which implies the product formula $\whp_{{}_G}(z)=\prod_{k=1}^{d(v_0)}\whp_{{}_{G_k}}(z)$ in \eqref{hathat} right away. 
 
 To begin the proof of the induction step for the first equality in \eqref{eqlts}, we note that $\psi_{{}{\widehat{G}_k}}(\pm 1)=0$ by \cite[Lemma 1.7(iv)]{FCh} since $\widehat{G}_k$ is a tree. This, \eqref{hathat}  and the induction assumption applied to $\widehat{\psi}_{{}_{\widehat{G}_k}}$ imply
 \begin{equation*}\begin{split}
 \widehat{\psi}_{{}_G}(\pm1)&=\prod\nolimits_{k=1}^{d(v_0)}\big(0+(\mp1)\cdot
 \widehat{\psi}_{{}_{\widehat{G}_k}}(\pm1)\big)\\&=\prod\nolimits_{k=1}^{d(v_0)}(\mp1)\cdot
 (\mp1)^{p_{{}_{\widehat{G}_k}}-1}=(\mp1)^{\sum\nolimits_{k=1}^{d(v_0)}p_{{}_{\widehat{G}_k}}}=(\mp1)^{p_{{}_G}-1},\end{split}
 \end{equation*}
 as required in the first equality in \eqref{eqlts}.
 
 To begin the proof of the induction step for the second equality in \eqref{eqlts}, we fix a natural $d\in[1,d(v_0)]$ and split $G=G'\cup G''$ where $G'=\cup_{k=1}^dG_k$ and $G''=\cup_{k=d+1}^{d(v_0)}G_k$. Writing $zd(v_0)=zd+z(d(v_0)-d)$ and using \eqref{hathat2} for $G'$ and $G''$ yields the decomposition $\psi_{{}_G}(z)/\whp_{{}_G}(z)=\psi_{{}_{G'}}(z)/\whp_{{}_{G'}}(z)+\psi_{{}_{G''}}(z)/\whp_{{}_{G''}}(z)$. In turn, multiplying this by $\whp_{{}_G}(z)$ and dividing by $(z^2-1)$,  the product formula in \eqref{hathat} yields $\wtp_{{}_G}(z)=\wtp_{{}_{G'}}(z)\cdot\whp_{{}_{G''}}(z)+\whp_{{}_{G'}}(z)\cdot\wtp_{{}_{G''}}(z)$. Using the last formula, the induction assumption for $G'$ and $G''$, and the first equality in \eqref{eqlts} yields the required assertion for $G$, 
 \begin{equation*}
 \begin{split}
 \wtp_{{}_G}(\pm1)&=(p_{{}_{G'}}-1)(\mp1)^{(p_{{}_{G'}}-2)}\cdot(\mp1)^{p_{{}_{G''}}-1}
 +(\mp1)^{p_{{}_{G'}}-1}\\&\cdot (p_{{}_{G''}}-1)(\mp1)^{(p_{{}_{G''}}-2)}=(p_{{}_G}-1)(\mp1)^{p_{{}_G}-2},
 \end{split}
 \end{equation*}
 because $p_{{}_{G'}}+p_{{}_{G''}}=p_{{}_G}+1$ as $v_0$ in $G'$ and $G''$ is counted twice.
 \end{proof}

 \begin{remark}\label{rem:NEWForm}
 In what follows we will need yet another version of formula \eqref{2.10} that involves only linearly independent functions $\phi_i$, $\phi_{i_1i_2}$, etc. To this end, assuming hypotheses imposed in Theorem \ref{Theorem 2.2}, we will need some notation. For each $r\in\{1,\ldots,p\}$ we introduce notation 
 \begin{equation}\label{defcLr}
 \mathcal{L}_r:=\text{Linear Span}\{\phi_{i_1\dots i_r}(\cdot): 1\le i_1<\ldots<i_r\le p\}
 \end{equation}
 for the real linear span of the characteristic functions $\phi_{i_1\dots i_r}(\cdot)$.
 We denote by $R_r$ the dimension of $\mathcal{L}_r$  so that $R_1$ is the dimension of the subspace of linear combinations $\sum_{i=1}^p\theta_i\phi_i(\cdot)$ of the functions $\phi_i(\cdot)$ with real $\theta$'s, $R_2$ of $\sum_{1\le i_1<i_2\le p}\theta_{i_1i_2}\phi_{i_1i_2}(\cdot)$, and so on, and $R_p=1$, the dimension of the span of $\phi_{1\ldots p}(\cdot)$. Furthermore, we denote by $\varphi_{r,l}(\cdot)$, $l=1,\ldots,R_r$, the basis of $\mathcal{L}_r$. Thus, each $\varphi_{r,l}(\cdot)$ is a linear combination of the functions $\phi_{i_1\ldots i_r}(\cdot)$ with multi-indices $i_1\ldots i_r$ and real coefficients, and, in particular, 
 for each $1\le j_1<\ldots<j_r\le p$ one has $\phi_{j_1\ldots j_r}(\cdot)
 =\sum_{l=1}^{R_r}\theta^{(l)}_{j_1\ldots j_r}\varphi_{r,l}(\cdot)$ with a unique choice of real coefficients $\theta$'s. Using these formulas for $r=1,2,\ldots,p$ in \eqref{2.10} yields its new version,
 \begin{equation}
\label{NEW2.10}\begin{split}
\phi(\lambda,&b_1,\ldots,b_p)=\phi(\lambda,0,\ldots,0)+
\sum_{l=1}^{R_1}B_{1,l}\varphi_{1,l}(\lambda)+\sum_{l=1}^{R_2}
B_{2,l}\varphi_{2,l}(\lambda)+\ldots\\&+
\sum_{l=1}^{R_r}
B_{r,l}\varphi_{r,l}(\lambda)+\ldots+B_{p,1}\varphi_{p,1}(\lambda),
\end{split}
\end{equation}
where the coefficients $B_{r,l}$, $l=1,\ldots,R_r$, $r=1,\ldots,p$, and $b_i$, $i=1,\ldots,p$, are related as follows,
\begin{equation}\label{NEW2.10new}
\begin{split}
B_{r,l}&=\sum_{1\le i_1<\ldots<i_r\le p} b_{i_1}\cdot
\ldots\cdot b_{i_r}\theta^{(l)}_{i_1\ldots i_r}.\end{split}
\end{equation}
One can choose $\varphi_{r,l}(\cdot)=\phi_{i_1^{(l)}\ldots i_r^{(l)}}(\cdot)$ with some multi-index $i_1^{(l)}\ldots i_r^{(l)}$. With this choice of $\varphi_{r,l}(\cdot)$ formula \eqref{NEW2.10new} becomes
\begin{equation}\label{NEW2.10new2}
\begin{split}
B_{r,l}&=b_{i_1^{(l)}}\cdot\ldots\cdot b_{i_r^{(l)}}+
\sum_{i_1\ldots i_r} b_{i_1}\cdot
\ldots\cdot b_{i_r}\theta^{(l)}_{i_1\ldots i_r}\end{split}
\end{equation}
where the summation is over all multi-indices $i_1\ldots i_r$ not equal to $i_1^{(l)}\ldots i_r^{(l)}$.
 For future use we introduce notation $R=R_1+\ldots+R_p$ and note that $p\le R\le 2^p-1$. The assertions are easily derived from the lemma given next that shows that the characteristic  functions at different ``levels'' are linearly independent.
  \hfill$\Diamond$\end{remark}.
 \begin{lemma}\label{Lemma 4.1} Under the assumptions in Theorem \ref{Theorem 2.2}, for any $r,s\in\{1,\ldots,p\}$ such that $s>r$ the linear spans $\mathcal{L}_r$ and $\mathcal{L}_s$ introduced in \eqref{defcLr} satisfy
 $\mathcal{L}_r\cap\mathcal{L}_s=\{0\}$. 
\end{lemma}
\begin{proof} 
Suppose that 
$\sum_{j_1\ldots j_s}\beta_{j_1\ldots j_s}\phi_{j_1\ldots j_s}
\equiv\sum_{i_1\ldots i_r} \alpha_{i_1\ldots i_r}\phi_{i_1\ldots i_r}$
with some coefficients $\alpha,\beta$'s.
Due to \eqref{2.12}--\eqref{2.15} this is equivalent to
\begin{equation}
\label{@}
\left(s(\lambda, \ell)\right)^{s-r}\sum_{j_1\ldots j_s}\beta_{j_1\ldots j_s}\psi_{j_1\ldots j_s}(c( \lambda, \ell))
\equiv\sum_{i_1\ldots i_r} \alpha_{i_1\ldots i_r}\psi_{i_1\ldots i_r}(c( \lambda, \ell)).
\end{equation}
The right-hand side is a polynomial of $c(\lambda, \ell)$. We denote it by $P(c(\lambda,\ell))=C\prod(c(\lambda,\ell)-\gamma_k)$ where $C$ is a constant and $P(\gamma_k)=0$. Choose $\lambda_0$ such 
that $\sin\sqrt{\lambda}_0\ell\not=0$ and $\cos\sqrt{\lambda_0}\ell\not=\gamma_k$.
Taking into account the following relations, cf.\ \cite[Corollary 12.2.10]{MP0},
\begin{align}\label{*}
c(\lambda,\ell)&=\cos\sqrt{\lambda}\ell+Q\frac{\sin\sqrt{\lambda}\ell}{\sqrt{\lambda}}+\frac{\tau_1}{\sqrt{\lambda}}, \\
\label{**}
s(\lambda, \ell)&=\frac{\sin\sqrt{\lambda}\ell}{\sqrt{\lambda}}-Q\frac{\cos\sqrt{\lambda}\ell}{\lambda}+\frac{\tau_2}{\lambda},
\end{align}
where $\tau_i$ are certain functions of Paley-Wiener class 
and $Q=\frac{1}{2}\int_0^{\ell}q(x)dx$, we conclude that for $\sqrt{\lambda}_n=\sqrt{\lambda}_0+{2\pi n}/{\ell}$
\[\lim_{n\to\infty}
\left(s(\lambda_n, \ell)\right)^{s-r}\sum_{j_1\ldots j_s}\beta_{j_1\ldots j_s}
\psi_{j_1\ldots j_s}(c( \lambda_n,\ell))=0
\]
and
$\lim_{n\to\infty}P(c(\lambda_n,\ell))=P(\cos\sql_0\ell)$.
Thus both functions in \eqref{@} must be identically zero as required.
\end{proof}

 \section{Eigenvalue asymptotics}

In this section, $G$ is an equilateral compact tree with $g$ edges and $p=g+1$ vertices, and we consider the Robin problem  \eqref{2.3}--\eqref{2.6} with zero potential, $q_j(x)=0$, $x\in[0,\ell]$, $j=1,\ldots,g$. We recall that $z_l$ are zeros of the polynomial $\psi(z)=\det(-zD+A)$ from \eqref{2.2} with multiplicities $m_l$, $l=1,\ldots,p$, where $z_1=-1$, $z_p=1$ and $m_1=m_p=1$, and that $\wtp(z)=(z^2-1)^{-1}\psi(z)$.

\begin{theorem}\label{Theorem 3.1}  Let $G$ be an equilateral compact tree and assume that the potential is identically zero. 
The spectrum $\{\lambda\}$ of the Robin boundary value problem (\ref{2.3})--(\ref{2.6}) is the union of $2p-3$ sequences,
\begin{equation}\label{lambdaseq}
\big\{\lambda_k^{(1)}\big\}_{k=0}^\infty,\,
\big\{\lambda_k^{(l)}\big\}_{k=0}^\infty,\,
\big\{\lambda_k^{(-l)}\big\}_{k=1}^\infty,\, l=2,\ldots,p-1,
\end{equation}
with the following asymptotics,
\begin{align} 
\label{3.2}
\sqrt{\lambda_k^{(1)}}&=(\pi k)/\ell+o(1), \text{ as $k\to+\infty$},\\ 
\label{3.1}
\sqrt{\lambda_k^{(\pm l)}}&=\big(\pm\arccos z_{l}+2\pi k\big)/\ell+o(1) \text{  for $ l=2,..., p-1$}.
\end{align}
\end{theorem}

\begin{proof} It is convenient to view $\sql$ as a new complex variable and introduce the functions 
\begin{equation}\label{deffg}\begin{split}
f(\sql)&=-\sql\sin(\sql\ell)\wtp\big(\cos(\sql\ell)\big) \, \text{ and } \\ 
g(\sql)&=
\sum\nolimits_{i=1}^p b_i\phi_{i}(\lambda)+\sum\nolimits_{1\leq i_1<i_2\leq p}b_{i_1}b_{i_2}\phi_{i_1i_2}(\lambda)+
\ldots \\ &\hskip3cm 
+ b_1\cdot\ldots\cdot b_p \phi_{12...p}(\lambda)\end{split}
\end{equation}
so that $f(\sql)=\phi(\lambda,0,\ldots,0)$ and $f(\sql)+g(\sql)=\phi(\lambda,b_1,\ldots,b_p)$, cf.\ \eqref{2.11}. In particular, the eigenvalues $\lambda$'s of the Robin problem are the squares of the zeros $\sql$'s of $f+g$. The functions $f$ and $g$ are even entire functions. The set of zeros of the function $f$ is the union of the following $4p-6$ sequences,
\begin{equation}\label{zerosf}\begin{split}
&\big\{\pm\pi k/\ell\big\}_{k=0}^\infty, \,
\big\{\pm(\arccos z_l+2\pi k)/\ell\big\}_{k=0}^\infty, \\
&\big\{\pm(-\arccos z_l+2\pi k)/\ell\big\}_{k=1}^\infty, \, l=2,\ldots,p-1,\end{split}
\end{equation}
where in the last sequences as well as in the last sequences in \eqref{lambdaseq} numeration starts with $k=1$, and $0$ is a double zero of $f$ since both $\sql$ and $\sin(\sql\ell)$ in the definition of $f$ are zeros at $\sql=0$.

To establish the asymptotic relations \eqref{3.2}--\eqref{3.1}, we will now show that for each $k$ large enough a small circle $\gamma_k$ centered at the $k$-th element of each of the sequences  \eqref{zerosf} contains exactly one zero of the function $\sql\mapsto f(\sql)+g(\sql)=\phi(\lambda,b_1,\ldots,b_p)$. To do that we will show that $|f(\sql)|>|g(\sql)|$ for all $\sql\in\gamma_k$ and apply Rouche's theorem.

To begin, we consider the first sequence in \eqref{zerosf} with the $+$-sign and let $\gamma_k$ denote the circle centered at $\pi k/\ell$ of radius $r=k^{-1/2}$ so that if $\sql\in\gamma_k$ then $\sql\ell=\pi k+r\ell e^{i\theta}$ for some $\theta\in[0,2\pi)$. Then $\cos(\sql\ell)=(-1)^k+o(r)$ and $|\sin(\sql\ell)|=cr+o(r)$ as $r\to0$; here and in what follows $c$ stands for a positive constant independent on parameters in equations that could change from one estimate to another. We recall that $\wtp(\pm1)\neq0$ because $\pm1$ are simple roots of $\psi$ and that $|\psi_i(\pm1)|=1$ by Lemma \ref{Lemma 2.5}. It follows that if $\sql\in\gamma_k$ then
\begin{equation}\label{fbel}
|f(\sql)|=|\sql|\cdot|\sin(\sql\ell)|\cdot|\wtp(\cos(\sql\ell))|\ge ck\cdot ck^{-1/2}\cdot c=ck^{1/2}
\end{equation}
and that, using \eqref{2.19}--\eqref{2.20new},
\begin{equation}\label{gabov1}\begin{split}
|\phi_i(\lambda)|&=|\psi_i(\cos(\sql\ell))|=|\psi_i((-1)^k+o(r))|\le c,\\
|\phi_{i_1i_2}(\lambda)|&\le ck^{-1}\cdot|\psi_{i_1i_2}(\cos(\sql\ell))|\le ck^{-1},\dots, 
|\phi_{1\ldots p}(\lambda)|\le ck^{-1}.
\end{split}
\end{equation}
Estimates \eqref{gabov1} yield $\big|\sum b_i\phi_i(\lambda)+\sum b_{i_1}b_{i_2}\phi_{i_1i_2}(\lambda)+\ldots+b_{1}\dots b_{p}\phi_{1\dots p}(\lambda)\big|\le c$
and so the inequality $|f(\sql)|\ge ck^{1/2}>c\ge|g(\sql)|$ warrants the desired application of the Rouche theorem.

To continue, we consider the second sequence in \eqref{zerosf} with the plus-sign and a fixed $l=2,\ldots,p-1$ and let $\gamma_k$ denote the circle centered at $(\arccos z_l+2\pi k)/\ell$ of radius $r=k^{-1/(2m_l)}$ so that if $\sql\in\gamma_k$ then $\sql\ell=\arccos z_l+2\pi k+r\ell e^{i\theta}$ for some $\theta\in[0,2\pi)$. Then $|\sql|\ge ck$, $|\sin(\sql\ell)|=|\sin(\arccos z_l+\ell re^{i\theta})|\ge c$ because $z_l\neq\pm1$, and $\cos(\sql\ell)=z_l+cre^{i\theta}+o(r)$ yielding $|(\cos(\sql\ell)-z_l)^{m_l}|\ge cr^{m_l}$. Since $m_l$ is the multiplicity of the zero $z_l$ of $\psi(z)$, we have
$\wtp(z)=(z-z_l)^{m_l}\wtp^{(l)}(z)$ where $\wtp^{(l)}$ is a polynomial such that $\wtp^{(l)}(z_l)\neq0$. Collecting all this together yields $|f(\sql)|\ge ck^{1/2}$ for $\sql\in\gamma_k$ as in \eqref{fbel}. The estimates as in \eqref{gabov1} are proved analogously, and thus Rouche theorem can be applied as above. All other cases in \eqref{zerosf} are similar.

In the course of proof of  \eqref{3.2},\eqref{3.1} we also proved that the zeros of the characteristic functions $\phi(\lambda,b_1,\ldots,b_p)$ and $\phi(\lambda,0,\dots,0)$ are in one-to-one correspondence provided they are large enough (that is, provided their numbers $k$ are large enough). The zeros of the respective characteristic functions are the eigenvalues of the eigenvalues problems with the Robin and standard conditions respectively. To finish the proof of the assertion in the theorem saying that the union of the sequences in \eqref{lambdaseq} gives the entire spectrum of the Robin problem, we need to show that for a sufficiently large $R$ the eigenvalues of the two problems located in the segment $[-R^2,R^2]$ are in one-to-one correspondence. To this end, we introduce yet another parameter, $t\in[0,1]$, and consider the characteristic function $\phi(\lambda, tb_1,\ldots,tb_p)$, a homotopy between $\phi(\lambda,b_1,\ldots,b_p)$ and $\phi(\lambda,0,\dots,0)$. When $t$ changes from $t=1$ to $t=0$ the Robin eigenvalues in the segment $[-R^2,R^2]$ move and eventually become the eigenvalues of the problem with the standard boundary conditions. Since the eigenvalues are real, the only possible ``loss'' of eigenvalues occurs when they ``leak'' through the end points of the segment. We may choose $R$ such that $\lambda=R^2$ is not a ``standard'' eigenvalue, equivalently, such that $f(R)\neq0$; here $f:\sql\mapsto\phi(\lambda,0,\ldots,0)$ as before.
Thus, to show that the number of the Robin eigenvalues in $[-R^2,R^2]$ is equal to the number of the ``standard'' eigenvalues in $[-R^2,R^2]$ we need to show the following assertion:  For all $t\in[0,1]$ the point $\sql=R$ is not a zero of the function $f+g_t:\sql\mapsto\phi(\lambda,tb_1,\ldots,tb_p)$ where $g_t(\sql)$ is defined as $g(\sql)$ in \eqref{deffg} except the Robin constants $b_i$ are replaced by $tb_i$, $i=1,\ldots,p$. In turn, the required assertion holds provided we are able to choose a large enough $R$ so that $|f(R)|>|g_t(R)|$ uniformly for all $t\in[0,1]$. This task is similar (but simpler) to what we have just accomplished justifying applicability of the Rouche theorem.

To proceed, we choose $R=(2\pi k+k^{-1/2})/\ell$ for $k$ large enough. As in \eqref{fbel}, \eqref{gabov1} we will use that $|\psi_i(1)|=1$ and $\wtp(1)\neq0$ by Lemma \ref{Lemma 2.5}. Thus, for $\sql=R\ge ck$ we write
\begin{align*}
\cos(R\ell)&=\cos(2\pi k+k^{-1/2})\ge c, \\ &|\wtp(\cos(R\ell))|\ge c,\,
|\sin(R\ell)|=|\sin(k^{-1/2})|\ge ck^{-1/2},\\
|\phi_i(\lambda)|&=|\psi_i(\cos(R\ell))|=|\psi_i(1+O(k^{-1/2}))|\le c,\\ &
|\phi_{i_1i_2}(\lambda)|\le c/k, \ldots, |\phi_{1\dots p}(\lambda)|\le c/k.
\end{align*}
This leads to the required inequality $|f(R)|\ge ck^{1/2}>c\ge|g(R)|$.
\end{proof}

We now provide a more refined asymptotics for $\{\lambda_k^{(1)}\}_{k=0}^\infty$ from \eqref{lambdaseq}.


\begin{theorem}\label{Theorem 3.2}
Let $G$ be an equilateral tree with identically zero potential. Then the asymptotics 
in  (\ref{3.2})  of the  sequence $\{\lambda_k^{(1)}\}_{k=0}^\infty$ of the eigenvalues of the Robin problem \eqref{2.3}--\eqref{2.6} can be refined as follows,
\begin{equation}
\label{3.11}
\sqrt{\lambda_k^{(1)}}=(\pi k)/\ell+\big((\pi(p-1))^{-1}\sum\nolimits_{i=1}^pb_i\big)(1/k)+o({1}/{k})
\text{ as $ k\to +\infty$.}
\end{equation}
\end{theorem}
\begin{proof} We will be looking for a constant $\alpha$ such that the expression 
\begin{equation}\label{dfalp}
\sql=\pi k/\ell+\alpha\ki+o(\ki) \text{ as $k\to\infty$}\end{equation} satisfies the equation $\phi(\lambda,b_1,\ldots,b_p)=0$. One can re-write conclusions of Lemma \ref{Lemma 2.5} as $\wtp((-1)^k)=(p-1)(-1)^{(k+1)p}$ and $\psi_i((-1)^k)=(-1)^{(k+1)(p-1)}$. This and \eqref{2.19}--\eqref{2.20new} yield, 
\begin{align}\nonumber
\sin(\sql\ell)&=\sin\big(\pi k+(\alpha\ell\ki+o(\ki))\big)\\
&\qquad=(-1)^k(\alpha\ell\ki+o(\ki))+o(\ki),\label{39}\\\nonumber
\cos(\sql\ell)&=\cos\big(\pi k+\alpha\ki+o(\ki)\big)=(-1)^k+o(\ki),\\\label{311}
\wtp(\cos(\sql\ell))&=\wtp((-1)^k+o(\ki))=(p-1)(-1)^{(k+1)p}+o(\ki),\\\label{312}
\phi_i(\lambda)&=(-1)^{p-1}\psi_i(\cos(\sql\ell))=(-1)^{p-1}\psi_i((-1)^k+o(\ki))\nonumber\\&=(-1)^{(k+2)(p-1)}+o(\ki),\\\label{313}
\phi_{i_1i_2}(\lambda)&=O(\ki), \ldots, \phi_{1\ldots p}(\lambda)=O(\ki).
\end{align}
Plugging \eqref{dfalp}, \eqref{39}, \eqref{311} into $f(\sql)=-\sql\sin(\sql\ell)\wtp\big(\cos(\sql\ell)\big)$ gives
\[f(\sql)=(-1)^{(k+1)(p+1)}\alpha\pi(p-1)+o(1)\] and by \eqref{312},\eqref{313} the equation
\[\phi(\lambda,b_1,\ldots,b_p)=(-1)^{p}f(\sql)+\sum_{i=1}^pb_i\phi_i(\lambda)+O(\ki)=0\] reads
\[(-1)^{(k+1)(p+1)+p}\alpha\pi(p-1)+o(1)+(-1)^{(k+2)(p-1)}\sum\nolimits_{i=1}^pb_i=0.\] Solving for $\alpha$ and using \eqref{dfalp} implies \eqref{3.11}.
\end{proof}

\section{An Inverse Problem} 
In this section, $G$ is again an  equilateral tree with $g$ edges and $p=g+1$ vertices, and we consider the Robin problem  \eqref{2.3}--\eqref{2.6} with zero potential.  
We assume that the shape of the graph is known as well as the potential $q(x)$  and, in particular, the matrices $D$ and $A$ and thus the determinants  $\psi(z)$, $\psi_i(z)$, $\psi_{i_1i_2}(z)$, and so on, from \eqref{2.2}, are given. Consequently, we may determine the functions $\phi(\lambda,0,\ldots,0)$, $\phi_i(\lambda)$, $\phi_{i_1i_2}(\lambda)$, and so on, using equations \eqref{2.18}--\eqref{2.20new}. 

In Theorem \ref{Theorem 4.2} we will consider the inverse problem of finding the coefficients $b_i$, $i=1,\ldots,p$, in the Robin conditions at the vertices provided we are given some spectral data, specifically, provided we are given distinct eigenvalues $\lambda_m$ of the Robin problem so that $\phi(\lambda_m,b_1,\ldots,b_p)=0$ for $m=1,\ldots, R$ with $R=R_1+\ldots+R_p$   defined in the end of Remark \ref{rem:NEWForm}. 
We recall that 
the functions $\varphi_{r,l}(\lambda)$,  $l=1,\ldots,R_r$, $r=1,...,p$ introduced in Remark \ref{rem:NEWForm} are linearly independent.  Given $\lambda_1,\ldots,\lambda_{R-1}$ and any $\lambda$, we introduce notation
$\varphi(\lambda_1,\ldots\lambda_{R-1},\lambda)$ for the following determinant, 
\begin{equation}\label{eq4.2}\begin{split}
\det\begin{bmatrix}\varphi_{1,1}(\lambda_1)&\dots
& \varphi_{1,R_1}(\lambda_1) &\varphi_{2,1}(\lambda_1)&\dots&\varphi_{p,1}(\lambda_1)\\
\varphi_{1,1}(\lambda_2)&\dots
& \varphi_{1,R_1}(\lambda_2) &\varphi_{2,1}(\lambda_2)&\dots&\varphi_{p,1}(\lambda_2)\\
\vdots&\vdots
&\vdots&\vdots&\vdots&\vdots\\
\varphi_{1,1}(\lambda_{R-1})&\dots
& \varphi_{1,R_1}(\lambda_{R-1}) &\varphi_{2,1}(\lambda_{R-1})&\dots&\varphi_{p,1}(\lambda_{R-1})\\
\varphi_{1,1}(\lambda)&\dots
& \varphi_{1,R_1}(\lambda) &\varphi_{2,1}(\lambda)&\dots&\varphi_{p,1}(\lambda)
\end{bmatrix},\end{split}
\end{equation}
so that $\varphi(\lambda_1,\ldots\lambda_{R-1},\lambda_{R})$ is the determinant of the system 
\begin{equation}\label{4.1}
\begin{split}
\sum\nolimits_{l=1}^{R_1} B_{1,l}\varphi_{1,l}(\lambda_m)
&+\sum\nolimits_{l=1}^{R_2} B_{2,l}\varphi_{2,l}(\lambda_m)
+\ldots+ B_{p,1} \varphi_{p,1}(\lambda_m)\\&=-\phi(\lambda_m, 0,\ldots,0), \quad
m=1,\ldots, R.\end{split}
\end{equation}



 Our plan is to use the given distinct eigenvalues $\lambda_m$, $m=1,\ldots,R$, and known functions $\phi(\cdot,0\ldots,0)$, $\varphi_{1,l}(\cdot)$, $\varphi_{2,l}(\cdot)$, etc.,  to  utilize formula \eqref{NEW2.10} to form a non-homogeneous  system of $R$ linear algebraic equations with $R$ unknowns $B_{r,l}$, $l=1,\ldots, R_r$, $r=1,\ldots p$. 

To begin, we will need the notion of \textit{sine type function} taken from  \cite{LO} and several facts about the sine type functions, see \cite[Chapter 11]{MP0}.
\begin{definition}\label{Definition 4.1}  
An entire function $\omega:\sqrt{\lambda}\mapsto \omega(\sql)$ of exponential type $\sigma>0$ is said to be of \textit{sine type} provided the following holds:\,
(1) all zeros of $\omega$ lie in a horizontal strip $|\im\sqrt{\lambda}|<h$,
(2) for a certain fixed value $\im\sqrt{\lambda}=h_1$ and positive constants $M_1,M_2$ one has
$M_1\leq |\omega(\re\sqrt{\lambda}+ih_1)|\leq M_2$ for all $\re\sql\in(-\infty,\infty)$,
and (3) the types of $F$ in the upper and in the lower half-planes are equal.
\end{definition}
A  product of sine type functions is again a sine type function.
 The functions $\sql\mapsto\sin\sql\ell$, $\cos\sql\ell$ and (nonconstant) polynomials of $\cos\sql\ell$ with real coefficients, see Lemma \ref{Lemma 4.2} below, are of sine type while the functions $\sql\mapsto\sql\sin\sql\ell$ and $\sin(\sql\ell)/\sql$ are not of sine type due to the factors $\sql$ and $1/\sql$; the function $\sql\mapsto \textrm{const}$ is not a sine type either.
Any sine-type function can be presented in the form 
\begin{equation}
\label{4.2}
\omega(\sqrt{\lambda})=c \lim\nolimits_{n\to +\infty}\prod\nolimits_{k=-n}^n\big(1-{\sqrt{\lambda}}/{\sqrt{\lambda_k}}\big)
\end{equation} 
where $c\neq0$ is a constant and $\sqrt{\lambda}_k$ are zeros of $\omega$; here, if $\sqrt{\lambda}_k=0$ then the corresponding factor $\big(1-\sqrt{\lambda}/{\sqrt{\lambda}_k}\big)$ should be replaced by $\sqrt{\lambda}$.
We recall the following equivalent reformulation of Definition \ref{Definition 4.1} that goes back to \cite{LO}, see also \cite[Proposition 11.2.19]{MP0}.
\begin{lemma}\label{Definition 4.2}
An entire function $\omega:\sqrt{\lambda}\mapsto\omega(\sql)$ of exponential type $\sigma>0$ is of sine type if and only if there exist positive constants $m$, $M$ and $h$ such that
$m\leq |\omega(\sqrt{\lambda})|e^{-\sigma |\im\sqrt{\lambda}|}\leq M$
for all $\sql$ satisfying $|\im\sql|>h$.
\end{lemma} 

We stress that in the next lemma the polynomial $P_n$ is not constant.
\begin{lemma}\label{Lemma 4.2}
Let $P_n(z)=a_nz^n+a_{n-1}z^{n-1}+...+a_0$ be a polynomial of degree $n\geq 1$ with real coefficients. Then $\sql\mapsto P_n(\cos\sqrt{\lambda}\ell)$ is a sine type function.
\end{lemma}
\begin{proof}
Using Lemma \ref{Definition 4.2} it is easy to see that the function $\sql\mapsto\cos\sqrt{\lambda}\ell$ is
a function  of exponential type $\ell$ and is a sine type function;
the same is true for $\sql\mapsto(\cos\sqrt{\lambda}\ell-c)$ where $c$ is a constant. The product 
$P_n(\cos\sqrt{\lambda}\ell)=a_n\prod_{i=1}^n(\cos\sqrt{\lambda}\ell-z_i)$, where $z_i$ are the zeros of $P_n$,  is therefore a function of exponential type $n\ell$ and is a sine type function. 
\end{proof}

\begin{lemma}\label{Lemma 4.2new}
Let $\omega_1$ and $\omega_2$ be entire functions of $\sql$ of equal exponential  types $\sigma>0$. Also, we assume that $\omega_1$ and $\omega_2$ are of sine type. Furthermore, let us assume that 
the derivatives of $\omega_2$ satisfy $\omega_2(0)=\omega_2'(0)=\ldots=\omega^{(n-1)}_2(0)=0$ for some $n=1,2,\ldots$ so that the function $\omega:\sql\mapsto\omega_1(\sql)+(\sql)^{-n}\omega_2(\sql)$ is entire.
Then the function $\omega$ is also of sine type.
\end{lemma}
\begin{proof}
 Using Lemma \ref{Definition 4.2} we find  positive constants $m_i,M_i$ and $h_i$ such that $m_i\leq |\omega_i(\sqrt{\lambda})| e^{-\sigma |\im \sqrt{\lambda}|} \leq M_i$ for all $\sqrt{\lambda}$
 so that $|\im\sqrt{\lambda}|>h_i$  and $i=1,2$.     
 Choose $h>\max\{h_1,h_2, (M_2/m_1)^{1/n}\}$. Then, for $|\sqrt{\lambda}|\ge|\im\sqrt{\lambda}|>h$,
\begin{equation*}\begin{split}
 \big|\omega_1(\sqrt{\lambda})&+(\sqrt{\lambda})^{-n}\omega_2(\sqrt{\lambda})\big|
 e^{-\sigma|\im\sqrt{\lambda}|}\leq M_1+|\sqrt{\lambda}|^{-n}M_2\leq M_1+M_2/h^n,\\       
 \big|\omega_1(\sqrt{\lambda})&+(\sqrt{\lambda})^{-n}\omega_2(\sqrt{\lambda})\big|
 e^{-\sigma|\im\sqrt{\lambda}|}\geq m_1-|\sqrt{\lambda}|^{-n}M_2 \geq m_1-M_2/h^n,
 \end{split}
 \end{equation*} 
 as required in Lemma \ref{Definition 4.2} to check that $\omega$ is of sine type. 
\end{proof}



We recall notation $\phi(\lambda, b_1,\ldots,b_p)$ for the characteristic function of the Robin problem, formula \eqref{NEW2.10}  relating the characteristic function and functions 
$\varphi_{r,l}(\cdot)$, $l=1,\ldots,R_r$, $r=1,\ldots,p$ introduced in Remark \ref{rem:NEWForm}, notations
$\phi$, $\phi_i$, $\phi_{i_1i_2}$,\ldots, $\phi_{1\ldots p}$, notations \eqref{2.2} for the functions $\psi_i$, $\psi_{i_1i_2}$, \ldots, $\psi_{1\ldots p}$ and formulas \eqref{2.18}--\eqref{2.20new} expressing $\phi$'s via $\psi$'s.  We introduce polynomials $\psi_{r,l}$ for  $l=1,\ldots,R_r$, $r=1,\ldots,p$, so that, cf.\ \eqref{2.20bis}, equality
\begin{equation}\label{defpsirl}
\varphi_{r,l}(\lambda)=(-1)^{p-r}(\sin(\sql\ell)/\sql)^{r-1}\psi_{r,l}(\cos(\sql\ell))
\end{equation}
holds.

\begin{lemma}\label{lemmaS} Let $G$ be an equilateral compact tree and
 $\sqrt{\lambda_1}$ be a zero of $\phi(\lambda, b_1, \dots b_p)$. Then the function $\sql\mapsto(\sqrt{\lambda}-\sqrt{\lambda_1})^{-1}\phi(\lambda, b_1, \dots ,b_p)$
is of sine type. \end{lemma}

\begin{proof}
Lemma \ref{Lemma 4.2} and \eqref{2.18} show that the function $(\sqrt{\lambda})^{-1}\phi(\lambda,0, \dots, 0)$ is of sine type as well as the function 
\begin{align}\label{eq44}
\omega&(\sql):=(-1)^{p-1}\sum_{l=1}^{R_1}B_{1,l}\psi_{1,l}(\cos(\sql\ell))\\
&+(-1)^{p-2}(\sql)^{-1}\sin(\sql\ell)
\sum_{l=1}^{R_2}B_{2,l}\psi_{2,l}(\cos(\sql\ell))+\ldots \notag\\
&\quad+(-1)^{p-r}(\sql)^{-(r-1)}(\sin(\sql\ell))^{r-1}
\sum_{l=1}^{R_r}B_{r,l}\psi_{r,l}(\cos(\sql\ell))\notag\\
&\qquad+
\ldots+B_{p,1}(\sqrt{\lambda})^{-p+1}(\sin(\sqrt{\lambda}l))^{p-1}.\notag
\end{align}
According to Lemma \ref{Definition 4.2} there exist positive numbers $h_1$, $m_1$, $M_1$, $h_2$, $m_2$, $M_2$ such that 
$m_1\leq |(\sqrt{\lambda})^{-1}\phi(\lambda,0,\dots,0)|e^{-\sigma |\im\sqrt{\lambda}|}\leq M_1$
provided $|\im \sqrt{\lambda}|>h_1$ and
$m_2\leq |\omega(\lambda,0,\dots,0)|e^{-\sigma |\im\sqrt{\lambda}|}\leq M_2
$
provided $|\im \sqrt{\lambda}|>h_2$.
Clearly,  
$
m_3\leq \left| \sqrt{\lambda}(\sqrt{\lambda}-\sqrt{\lambda_1})^{-1}\right|\leq M_3
$
provided $|\im\sqrt{\lambda}|>h_3$ for some positive numbers $h_3$, $m_3$ and $M_3$.
Since $\phi(\lambda,b_1,\ldots,b_p)=\phi(\lambda,0,\ldots,0)+\omega(\sql)$ by \eqref{NEW2.10},  \eqref{defpsirl} and \eqref{eq44}, we have
\[\left|\frac{\phi(\lambda, b_1, \dots ,b_p)}{
(\sqrt{\lambda}-\sqrt{\lambda_1})}\right|=\left|\frac{\sqrt{\lambda}}{\sqrt{\lambda}-\sqrt{\lambda_1}}\right|
|(\sqrt{\lambda})^{-1}(\phi(\lambda,0, \dots, 0)+\omega(\lambda))|
\]
and conclude that for $|\im \sqrt{\lambda}|>\max\{h_1,h_2,h_3\}$ one has 
\begin{align*}
m_3\big(m_1-{M_2}/{|\im\sqrt{\lambda}|}\big)&
\leq \left|({\sqrt{\lambda}})^{-1}\big({\phi(\lambda,0,\dots,0)+\omega(\sql)\big)}\right|e^{-\sigma |\im\sqrt{\lambda}|}\\&\leq M_3\big(M_1+{M_2}/{|\im\sqrt{\lambda}|}\big).
\end{align*}
We denote $h:=\max\{h_1,h_2,h_3, {M_2}/({m_1-\epsilon})\}$ for any pre-fixed  $0<\epsilon<m_1$. Then for $|\im\sqrt{\lambda}|\geq {h}$ we have the estimate

\[
m_3\epsilon\leq \left|({\sqrt{\lambda}})^{-1}\big({\phi(\lambda,0,\dots,0)+\omega(\sql)\big)}\right|e^{-\sigma |\im\sqrt{\lambda}|}\leq M_3(M_1+m_1-\epsilon)
\]
and so $\sql\mapsto(\sqrt{\lambda}-\sqrt{\lambda_1})^{-1}\phi(\lambda, b_1, \dots ,b_p)$ is a sine type function by Lemma \ref{Definition 4.2}.
\end{proof}


We are ready to present the last main result of this paper that describes the  spectral data in the inverse problem needed to recover the coefficients $B_{r,l}$ related to the Robin boundary conditions.
We recall notation $R=R_1+\ldots+R_p$ introduced in Remark \ref{rem:NEWForm} 

\begin{theorem}\label{Theorem 4.2} Let $G$ be an equilateral compact tree 
with $g$ edges and assume that the potential $q_j$ is identically zero for all $j=1,\ldots,g$.  Let $\lambda_m$, $m=1,\ldots,R-1$, be any $R-1$ given distinct eigenvalues of the Robin problem \eqref{2.3}--\eqref{2.6}. Then there exists yet another eigenvalue, $\lambda_{R}$, of the Robin problem such that the determinant in \eqref{eq4.2} satisfies $\varphi(\lambda_1,\ldots,\lambda_{R-1},\lambda_{R})\neq0$. Therefore, the coefficients $B_{r,l}$, $l=1,\ldots,R_r$, $r=1,\ldots,p$, defined in Remark  \ref{rem:NEWForm} can be uniquely expressed from system \eqref{4.1} via $\phi(\lambda_m,0,\ldots,0)$, $\varphi_{r,l}(\lambda_m)$ for $l=1,...,R_r$, $r=1,...,p$, $m=1,\ldots,R$).
\end{theorem}
\begin{proof}
Expanding the determinant in \eqref{eq4.2} by the last row, we conclude that the function $\sql\mapsto\varphi(\lambda_1,\ldots,\lambda_{R-1},\lambda)$ is a  linear combination (with real coefficients that we will denote by 
$\alpha_{r,l}(m)$, $l=1,...,R_r$, $r=1,...,p$, $m=1,\ldots,R$) of the functions 
$\sql\mapsto\varphi_{r,l}(\lambda)$, $l=1,...,R_r$, $r=1,...,p$)
 computed via $\psi_1,\ldots,\psi_{1\ldots p}$ in \eqref{2.19}--\eqref{2.20new} such that, in particular, equality \eqref{defpsirl} holds. 
 
 
 We introduce notation  $\omega_c(\sql)=\phi(\lambda,b_1,\ldots,b_p)$ for the characteristic function of the Robin problem,
  $\omega_d(\sql)=\varphi(\lambda_1,\ldots,\lambda_{R-1},\lambda)$ for the determinant \eqref{eq4.2},  and define functions $\omega_r$ for $r=1,\ldots,p$ by
 \begin{align}\label{eq49}
 \omega_r(\sql):=(-1)^{p-r}(\sin(\sql\ell))^{(r-1)}\sum_{l=1}^{R_r}\alpha_{r,l}(m)\psi_{r,l}(\cos(\sql\ell)), 
 \end{align}
 where the polynomials $\psi_{r,l}$ are defined in \eqref{defpsirl}.   
 Then 
 \begin{align}\label{eq410}
 \omega_d(\sql)=\omega_1(\sql)+(\sql)^{-1}\omega_2(\sql)
 +\ldots+(\sql)^{-(p-1)}\omega_p(\sql).
 \end{align}
Therefore, to prove Theorem \ref{Theorem 4.2} we need to find $\sqrt{\lambda_{R}}$ such that   
\begin{equation}\label{notzeros}
\text{$\sqrt{\lambda_{R}}$ is a zero of $\omega_c(\cdot)$ but not a zero of  $\omega_d(\cdot)$.}
\end{equation}
 Indeed, if this is the case then there is a Robin eigenvalue $\lambda_R$ such that, as required in the theorem, \begin{equation}\label{Ntz}\begin{split}
 \omega_c(\sqrt{\lambda_{R}})&=\phi(\lambda_{R}, b_1,\ldots,b_p)=0 \text{ but }\\ \omega_d(\sqrt{\lambda_{R}})&=\varphi(\lambda_1,\ldots,\lambda_{R-1},\lambda_{R})\neq0; \end{split}\end{equation}
 the last inequality and the structure of the determinant in \eqref{eq4.2} automatically implies that $\lambda_{R}$ differs from all other 
 $\lambda_{m}$, $m=1, \ldots, R-1$.

 Formula \eqref{eq49} and Lemma \ref{Lemma 4.2} show that $\omega_1$ is either a function of sine type or a constant (which could be zero), while $\omega_r$ for each $r=2,\ldots,p$ is either a function of sine type or identically equals to zero.
 
Let us first consider the case when $\omega_1$ is a function of sine type. Using \eqref{eq410} and Lemma \ref{Lemma 4.2new} (if needed, repeatedly), we conclude that $\omega_d$ is of sine type, and thus can be represented via its zeros as written in \eqref{4.2}.
However, from Lemma \ref{lemmaS} we know that $\omega_c$ is not of sine type. This implies that statement \eqref{notzeros} holds for a zero of $\omega_c(\cdot)$, which concludes the proof of the theorem in this case.

Let us now consider the case when $\omega_1(\sql)=c$ for all $\sql$ and $c\neq0$.
We take the sequence $\sqrt{\lambda^{(1)}_k}=(\pi k)/\ell+o(1)$ as $k\to\infty$ of zeros of 
the function $\omega_c$ described in equation \eqref{3.2} of Theorem \ref{Theorem 3.1}. It follows from \eqref{eq49} and \eqref{eq410} that $\lim_{k\to\infty}\omega_d(\sqrt{\lambda^{(1)}_k})=\lim_{k\to\infty}\omega_1(\sqrt{\lambda^{(1)}_k})=c\neq0$ and thus \eqref{Ntz} holds for $\lambda_R:=\lambda_k^{(1)}$ with $k$ large enough completing the proof in this case. 

Let us now consider the case when $\omega_1(\sql)=0$ for all $\sql$. Multiplying \eqref{eq410} by $\sql$ yields
$\sql\omega_d(\sql)=\omega_2(\sql)+\ldots+(\sql)^{-(p-2)}\omega_p(\sql)$. In the case when $\omega_2$ is not identically equal to zero the application of Lemma \ref{Lemma 4.2new} (if needed, repeated) shows that $\sql\mapsto\sql\omega_d(\sql)$ is of sine type. However, $\sql\mapsto\sql\omega_c(\sql)$ is not of sine type by Lemma \ref{lemmaS}, and thus \eqref{notzeros} follows from \eqref{4.2}. In the case when $\omega_2$ is identically equal to zero we multiply \eqref{eq410} by $(\sql)^2$, and apply the same argument iductively.
\end{proof}

{\bf Data Availability Statement:}

Data sharing is not applicable to this article as no datasets were generated
or analyzed during the current study.

{\bf Declarations}

The authors have no competing interests to declare that are relevant to the content of this article.

\end{document}